\documentclass[dvips]{degruyter-proceedings}  

\newcommand{\C}{\ensuremath{\mathbb{C}}}
\newcommand{\E}{\ensuremath{\mathrm{e}}}
\newcommand{\N}{\ensuremath{\mathbb{N}}}
\newcommand{\myO}{\ensuremath{\mathcal{O}}}

\newcommand{\R}{\ensuremath{\mathbb{R}}}
\newcommand{\myS}{\ensuremath{\mathcal{S}}}

\newcommand{\Z}{\ensuremath{\mathbb{Z}}}

\newcommand{\bfa}{\ensuremath{\mathbf a}}
\newcommand{\bfb}{\ensuremath{\mathbf b}}

\newcommand{\bfd}{\ensuremath{\mathbf d}}

\newcommand{\bfg}{\ensuremath{\mathbf g}}
\newcommand{\bfk}{\ensuremath{\mathbf k}}

\newcommand{\bfx}{\ensuremath{\mathbf x}}
\newcommand{\bfy}{\ensuremath{\mathbf y}}
\newcommand{\bfz}{\ensuremath{\mathbf z}}
\newcommand{\bfzero}{\ensuremath{\mathbf 0}}

\newcommand{\myone}{\mathbf{1}}

\newcommand{\myF}{\mathcal{F}}
\newcommand{\myint}{\ensuremath{\mathrm{int}}}

\newcommand{\etk}{Erd\"os-Tur\'an-Koksma}

\newcommand{\XOR}{\ensuremath{\texttt{XOR}}} 

\newcommand{\hyper}[1]{H_{#1}}
\newcommand{\hyperfamily}{\mathcal{H}_ C}

\newcommand{\F}{{\mathbb{F}}}

\DeclareMathOperator{\mydig}{dig}   

\title{The hybrid spectral test}

\lastnameone{Hellekalek}
\firstnameone{Peter}
\nameshortone{P.~Hellekalek}
\addressone{Dept. of Mathematics,\\
University of Salzburg, Hellbrunner Strasse 34, 5020 Salzburg}
\countryone{Austria}
\emailone{peter.hellekalek@sbg.ac.at}

\lastnametwo{}
\firstnametwo{}
\nameshorttwo{}
\addresstwo{}
\countrytwo{}
\emailtwo{}

\lastnamethree{}
\firstnamethree{}
\nameshortthree{}
\addressthree{}
\countrythree{}
\emailthree{}

\lastnamefour{}
\firstnamefour{}
\nameshortfour{}
\addressfour{}
\countryfour{}
\emailfour{}

\lastnamefive{}
\firstnamefive{}
\nameshortfive{}
\addressfive{}
\countryfive{}
\emailfive{}

\abstract{
The starting point of this paper is the interplay
between the construction principle of a sequence and the characters of the compact abelian group
that underlies the construction.
In case of the Halton sequence in base $\mathbf b=(b_1, \ldots, b_s)$ in the $s$-dimensional unit cube $[0,1)^s$,
which is an important type of a digital sequence,
this kind of duality principle leads to the so-called $\mathbf b$-adic function system
and provides the basis for the $\mathbf b$-adic method,
which we present in connection with hybrid sequences.
This method employs structural properties of the compact group of $\mathbf b$-adic integers
as well as $\mathbf b$-adic arithmetic to derive tools for the analysis of the uniform distribution of
sequences  in $[0,1)^s$.

We first clarify the point which function systems are needed to analyze digital sequences.
Then, we present the hybrid spectral test in terms of trigonometric-, Walsh-, and  $\mathbf b$-adic functions.
Various notions of diaphony as well
as many figures of merit for rank-1 quadrature rules in Quasi-Monte Carlo integration
and for certain linear types of pseudo-random number generators are included in this measure
of uniform distribution.
Further, discrepancy may be approximated arbitrarily close by suitable versions of the spectral test.
}

\keywords{Uniform distribution of sequences, discrepancy, diaphony, Fourier series, Walsh series, spectral test, pseudorandom number
generation, hybrid sequences}

\classification{11K06,11K38,11K45,11K70}

\researchsupported{}

\acknowledgments{My special thanks are due to Roswitha Hofer, University of Linz,
for several helpful and most enjoyable discussions.}

\firstpage{1}

\begin{document}

   \tableofcontents

\theoremstyle{plain}\newtheorem{Thm}{Theorem}[section]
\newtheorem{Prop}[Thm]{Proposition}
\newtheorem{Lem}[Thm]{Lemma}
\newtheorem{Conj}[Thm]{Conjecture}
\newtheorem{Cor}[Thm]{Corollary}
\newtheorem{Ex}[Thm]{Example}
\newtheorem{Def}[Thm]{Definition}
\newtheorem{Prb}[Thm]{Research Problem}
\newtheorem{Rem}[Thm]{Remark}

\newcommand{\gauss}[3]{\left[ { #1 \atop #2} \right]_{#3}}

\def\supp{\mathrm{supp\,}}
\newcommand{\wt}[1]{\mathrm{wt}(#1)}
\def\gdim{\mathrm{geomdim}}

\section{Introduction}
\label{s:introduction}

This paper is about certain mathematical concepts to analyze the uniform distribution behaviour
 of  sequences on the $s$-dimensional torus $[0,1)^s$.
 We discuss qualitative aspects, by which we understand the study of properties of an infinite sequence
 that induce its uniform distribution in $[0,1)^s$,
 and also quantitative aspects, i.e. the question how to measure the uniform distribution of a finite or
 infinite sequence on the $s$-torus.

A general setting to generate a (finite or infinite) sequence $\omega=(x_n)_{n\ge 0}$ in some output space
$\myO$, like $\myO=[0,1)^s$,
is the following.
We consider a nonempty set $\myS$,
the so-called state space, and a map $T:\myS \rightarrow \myS$,
the state update transformation.
We employ $T$ to generate a (finite or infinite) sequence of states $(s_n)_{n\ge 0}$ in $\myS$.
This sequence is then mapped to a sequence $\omega=(x_n)_{n\ge 0}$ in $\myO$ by an output map
 $\varphi: \myS\rightarrow \myO$,
 by letting $x_n=\varphi(s_n)$, $n\ge 0$.

In numerous applications of this concept,
the sequence $(s_n)_{n\ge 0}$ is constructed by iterating $T$.
We start with some initial state $s_0$ and put $s_n=T^n(s_0)$, $n\ge 0$.
Illustrative examples are linear or inversive congruential pseudorandom number generators (see \cite{Nie92a}),
the well known $(n\alpha)_{n\ge 0}$ sequences (see \cite{Kui74a}),
or the block cipher AES in Output Feedback mode (see \cite[p. 28]{Daemen02a}).

In some other cases, we employ a ring $(R,+, \cdot)$ with unity $1$ and a function
$\psi: R\rightarrow \myS$ to map the sequence $(n1)_{n\ge 0}$ (the so-called ``counter'')
first to a sequence
of states $(\psi(n1))_{n\ge 0}$ in $\myS$.
Then, this sequence is ``encrypted'' by $T$ to give the sequence
$(T\circ\psi(n1))_{n\ge 0}$ in $\myS$.
Finally, this sequence of states is mapped into the output space $\myO$.
This gives $\omega=(\varphi\circ T\circ\psi(n1))_{n\ge 0}$.
Examples of such constructions are explicit inversive congruential pseudorandom number generators
(see \cite{Hel95c,Nie10b,Nie11a,Topuzoglu07a}),
certain digital sequences (see \cite{Dick10a}),
or AES in Counter mode (see \cite[p. 28]{Daemen02a}).
In these examples, we let $R=\Z$, the integral domain of integers.

In all of the cases exhibited above,
we employ some arithmetical operation like addition on the state space $\myS$.
If $\myS$ is a compact abelian group with respect to the chosen operation,
then
we have  the arsenal of abstract harmonic analysis at our disposition (see \cite{Hewitt79a}).
The choice of the compact group $\myS$ determines which function system is suitable
for the analysis of the equidistribution behavior of the sequence $\omega$ in $\myO$,
because the group operation on $\myS$
is intrinsically related to the {\em dual group} $\hat{\myS}$ of $\myS$.

An example of such a suitable match between sequences and function systems
based on this \emph{duality principle} is given by Kronecker sequences or,
in the discrete version, good lattice points,
and the trigonometric functions.
Here,
$\myS=\myO=[0,1)^s$,
and the construction method uses
{\em addition modulo one} on the $s$-torus $[0,1)^s$,
$T(s)=s+\boldsymbol{\alpha}$, with $\boldsymbol{\alpha}\in\R^s$
(see \cite{Drm97a}, \cite[Ch. 5]{Nie92a} and \cite{Slo94a}).
The dual of the compact group $\myS$ may be interpreted as the trigonometric function system.
For the background of these group rotations in ergodic theory,
we refer to the monographs \cite{Parry04a,Walters82a}.

A second example of this duality principle is given by digital  nets and sequences and the Walsh functions.
Here,
{\em addition without carry} of digit vectors comes into play
(see \cite[Ch. 4]{Nie92a} and \cite{Dick10a}).
For example,
for nets and sequences in base 2,
the underlying group $\myS$ is the compact group $\F_2^\infty$
and its dual group is the Walsh function system in base $2$.
The same relation holds for general integer bases $b\ge 2$.

An important type of a digital sequence,
the Halton sequence in integer base $\mathbf b=(b_1, \ldots, b_s)$ on the $s$-torus $[0,1)^s$,
is generated by {\em addition with carry} of digit vectors.
For this reason,
it makes sense to choose as the underlying group $\myS$ the compact group of $\mathbf b$-adic integers
$\Z_{b_1}\times\dots\times\Z_{b_s}$.
Considering the dual group then leads to the $\bfb$-adic function system,
which is the main tool of the $\bfb$-adic method introduced in \cite{Hel09a,Hel10a,Hel10c}.
It will be discussed in Section \ref{s:notation}.

The general duality principle presented before also accommodates for \emph{hybrid sequences},
which are sequences $\omega$ of points in $[0,1)^s$ where
certain coordinates of the points stem from one lower-dimensional
sequence $\omega_1$,
with state space $\myS_1$,
and  the remaining coordinates from a second
lower-dimensional sequence $\omega_2$,
with state space $\myS_2$.
The duality principle leads us to consider \emph{hybrid function systems},
which arise from the product group $\hat{\myS}_1\times \hat{\myS}_2$.
Of course, this idea may be generalized to mixing more than two subsequences into
a hybrid sequence.

This paper is structured as follows.
In Section~\ref{s:addition} we consider the question which function systems will suffice to
analyze digital sequences.
In Section~\ref{s:notation} we introduce the necessary notation and
in Section~\ref{s:spectral} we define the hybrid spectral test and show that it is a measure of the
uniform distribution of a sequence in $[0,1)^s$.
Section~\ref{s:examples} deals with special cases of the spectral test, like diaphonies,
the spectral test for pseudorandom number generators, various figures of merit for integer lattice points,
like they appear in the context of good lattice points and rank-1 lattice rules.
Finally,
we show that the extreme as well as the star discrepancy can be approximated arbitrarily close by
special cases of the hybrid spectral test.

\section{Adding digit vectors}
\label{s:addition}

Digital sequences on the $s$-torus are sequences that arise from operations with digit vectors
in some given integer bases $b_i\ge 2$, $1\le i\le s$.
For the sake of simplicity,
we restrict the following discussion to the one-dimensional case.

As we have seen in Section~\ref{s:introduction},
and as the block ciphers IDEA~\cite{Lai91a} and AES~\cite{Daemen02a}
illustrate in cryptography,
two types of addition of  digit vectors are in use in applications,
addition without carry and addition with carry.
By the duality principle,
the first addition leads to Walsh functions and the second type to $b$-adic
functions
as the proper tools for the analysis of such sequences.

Are these two types of function systems sufficient to study digital sequences?
In algebraic terms,
are there any other possibilities to add digit vectors than addition with or without carry?
If yes, then this would lead to additional types of groups and dual groups and, hence,
to additional function systems.
If not, then the Walsh functions and the $b$-adic functions suffice to analyze
digital sequences.

For details, in particular for the proofs in the following considerations and a refinement
using compositions of positive integers instead of partitions and automorphisms of
certain groups to define additions,
we refer the reader to \cite{Hel12c}.

Let $b\ge 2$ be a fixed integer and let
$\mathcal{A}_b= \{ 0,1, \ldots, b-1\}$ denote the set of $b$-ary digits.
For $m\in \N$, $m\ge 2$,
let $\mathcal{A}_b^m$ stand for
the $m$-fold cartesian product of the set $\mathcal{A}_b$ with itself.
We study the following question:
What are the  binary operations ``$+$'' on the set $\mathcal{A}_b^m$ of  digit vectors of length $m$
such that the pair
$(\mathcal{A}_b^m, +)$ is an abelian group?

In this paper, when we speak of an ``addition on $\mathcal{A}_b^m$'',
we mean a binary operation ``$+$'' on the set $\mathcal{A}_b^m$ of  digit vectors in base $b$ such that the pair
$(\mathcal{A}_b^m, +)$ is an abelian group.
The reader should note that the term ``binary'' has two different meanings here,
which will become clear from the context.
A binary operation on a set $G$ is a map from the cartesian product $G\times G$ into $G$.
Referring to the representation of real numbers in base $b=2$,
the elements of the set $\mathcal{A}_2^m$ are called binary vectors,
and for $m=1$ one speaks of binary digits or bits.

Let us consider the case $b=2$ first.
There are two well known examples for addition of digit vectors,
which will be discussed below.

For $n\in \N$, $n\ge 2$,
let  $\Z/n\Z$ denote the additive group of residue classes modulo $n$.
We identify this cyclic group with the set of integers $\{0,1, \ldots, n-1\}$ equipped with addition modulo $n$.

\begin{example}
We identify $\mathcal{A}_2$ with $\Z/2\Z$.
For $\bfx, \bfy \in \mathcal{A}_2^m$,
$\bfx = (x_0, \ldots, x_{m-1})$ and $\bfy = (y_0, \ldots, y_{m-1})$,
we define
\[
\bfx + \bfy = (x_0\oplus y_0, \ldots, x_{m-1}\oplus y_{m-1}),
\]
where `$\oplus$' denotes addition on $\Z/2\Z$,
$0\oplus 0=1\oplus 1 =0$, and $0\oplus 1=1\oplus 0 =1$.
The pair $(\mathcal{A}_2^m, +)$ is an abelian group.
In fact,
it is isomorphic to the product group $(\Z/2\Z)^m$.
We call this binary operation {\em addition without carry},
or  {\em $\XOR$-addition} of digit vectors.
\end{example}

Every nonnegative integer $k$, $0\le k < 2^m$,
has a unique representation in base $2$ of the form $k=k_0+k_1 2 + \dots +k_{m-1} 2^{m-1}$
with digits $k_j\in \mathcal{A}_2$, $0\le j\le m-1$.

\begin{example}
We identify $\mathcal{A}_2^m$ with the group $\Z/2^m\Z$ as follows.
For $\bfx\in \mathcal{A}_2^m$,
$\bfx = (x_0, \ldots, x_{m-1})$ ,
we define the map $\myint_2: \mathcal{A}_2^m \rightarrow \Z/2^m \Z$,
\[
\myint_2(\bfx) = x_0+ x_1 2+\dots + x_{m-1}2^{m-1}.
\]
Further, let $\mydig_2: \Z/2^m \Z \rightarrow \mathcal{A}_2^m$,
\[
\mydig_2(k) = (k_0,  k_1, \ldots , k_{m-1}),
\]
where $k=k_0 + k_1 2 + \dots + k_{m-1}2^{m-1}$ is the representation of $k$ in base 2.
Finally,
for $\bfx, \bfy \in \mathcal{A}_2^m$,
we define
\[
\bfx + \bfy = \mydig_2 \left( \myint_2(\bfx) + \myint_2(\bfy) \pmod{2^m}\right).
\]

With this binary operation, the pair $(\mathcal{A}_2^m, +)$ is an abelian group.
Clearly,
it is isomorphic to the additive group $\Z/2^m \Z$.
We call this type of binary operation {\em  addition with carry} or {\em integer addition} of digit vectors.
\end{example}

For $m\ge 2$,
our two examples of addition act on non-isomorphic groups,
because $\Z/2^m\Z$ is cyclic and $(\Z/2\Z)^m$ is not.
Apart from these two examples,
are there any other possibilities to define addition on the set $\mathcal{A}_2^m$?
From the Fundamental Theorem for Finite Abelian Groups (see, for example,
\cite[Sec. 10]{Herstein99a}) we obtain the following lemma.
In this context, a {\em partition} \index{partition} of a positive integer $m$ is a
finite sequence $(t_i)_{i=1}^r$, $r\in \N$, of positive integers $t_i$
with the two properties
(i)  $t_1 \ge t_2 \ge  \dots \ge t_r$, and
(ii) $t_1+t_2+ \dots + t_r = m$.

\begin{lemma}
\label{ch:intro:cor:FTFAG}
The non-isomorphic groups of order $2^m$, $m\in \N$, are given by
the product groups
\[
(\Z/2^{t_1}\Z) \times (\Z/2^{t_2}\Z) \times \dots \times (\Z/2^{t_r}\Z),
\]
where $(t_i)_{i=1}^r$ is a partition of $m$.
\end{lemma}

Hence, in view of Lemma \ref{ch:intro:cor:FTFAG},
an addition on the set $\mathcal{A}_2^m$
is defined if we choose a partition
$m=t_1+ t_2 +\dots + t_r$ of $m$ and put
\begin{equation}\label{eqn:2-adic}
(\mathcal{A}_2^m, +) \cong (\Z/2^{t_1}\Z) \times (\Z/2^{t_2}\Z) \times \dots \times (\Z/2^{t_r}\Z).
\end{equation}
Here,
the symbol ``$\cong$'' denotes that the two groups are isomorphic.
From the structure of the factors in (\ref{eqn:2-adic}) we obtain the following information.

\begin{theorem}\label{cor:types}
The only two types of binary operations on (sub)vectors of digits that may appear in the group
law of the abelian group $(\mathcal{A}_2^m, +)$ are the following:
\begin{itemize}
\item addition given by finite product groups of the form
$(\Z/2\Z)\times \dots \times (\Z/2\Z)$,
which is what we have called addition without carry, or
\item addition in groups of residues of the form $\Z/2^t\Z$, $t\ge 2$,
which we have called addition with carry.
\end{itemize}
\end{theorem}

Theorem \ref{cor:types} directly generalizes from base 2 to arbitrary prime bases $p$.
In the case of a composite base $b$,
we also have the equivalent of  ( \ref{eqn:2-adic}),
as well as some additional direct products arising from the factorization of $b$ into prime powers.
Even in the latter cases, only two types of addition of (sub)vectors of digits appear,
addition with or without carry.

\begin{remark}
We have seen that there exist \emph{only two types} of addition for digit vectors,
while there exist many different variants of this binary operation,
for example by changing the underlying partition sequence.
As a consequence, from the duality principle  only two types of function systems arise,
the Walsh functions and the $b$-adic functions.
For those digits that are added without carry, the Walsh system applies,
and for those digits that are added with carry,  the $b$-adic system of functions is appropriate.
These two function systems cover all possible cases.
\end{remark}

\section{Notation}\label{s:notation}
Throughout this paper,
$b$ denotes a positive integer, $b\ge 2$,
and $\mathbf{b}=(b_1, \ldots, b_s)$ stands for a vector of not necessarily distinct
integers $b_i\ge 2$, $1\le i\le s$.
$\mathbb{N}$ represents the positive integers,
and we put $\mathbb{N}_0=\mathbb{N}\cup \{0\}.$

The underlying space is the $s$-dimensional torus
$\mathbb{R}^s/\mathbb{Z}^s$,
which will be identified with the half-open interval $[0,1)^s$.
Haar measure on the $s$-torus $[0,1)^s$ will be denoted by $\lambda_s$.
We put $e(y)=\E^{2\pi \mathrm{i} y}$ for $y\in\mathbb{R}$,
where $\mathrm{i}$ is the imaginary unit.

We will use the standard convention that empty sums have the value 0 and empty products
value 1.

For a nonnegative integer $k$, let
$k =\sum_{j\ge0} k_j\, b^j, k_j \in \{0,1,\ldots ,b-1\},$
be the unique $b$-adic representation of $k$ in base $b$. With the
exception of at most finitely many indices $j$, the digits $k_j$ are
equal to 0.

Every real number $x\in [0,1)$ has a  representation in base $b$
of the form
$x = \sum_{j\ge 0} x_j \,b^{-j-1},$
with digits
$x_j\in \{0,1, \ldots, b-1\}.$
If $x$ is a {\em $b$-adic rational},
which means that $x=ab^{-g}$, $a$ and $g$ integers, $0\le a< b^g$,
$g\in \mathbb{N}$,
and if $x\ne 0$,
then there exist two such representations.

The $b$-adic representation of $x$ is uniquely determined
under the condition that $x_j \ne b-1$ for infinitely many $j$.
In the following,
we will call this particular representation the {\em regular} ($b$-adic) representation
of $x$.

Let $\mathbb{Z}_b$ denote the compact group of the $b$-adic integers. We refer
the reader to \cite{Hewitt79a,Mahler81a} for details.
An element $z$ of $\mathbb{Z}_b$ will be written as a formal sum
$z=\sum_{j\ge 0} z_j\, b^j,$
with digits $z_j\in \{0,1,\ldots, b-1  \}$.
The set $\mathbb{Z}$ of integers  is embedded in $\mathbb{Z}_b$.
If $z\in \mathbb{N}_0$,
then at most finitely many digits $z_j$ are different from 0.
If $z\in \mathbb{Z}$, $z< 0$,
then at most finitely many digits $z_j$ are different from $b-1$.
In particular,
$-1 = \sum_{j\ge 0} (b-1)\, b^j.$

We recall the following concepts from \cite{Hel10a,Hel10c,Hel11a}.
\begin{definition}
The map $\varphi_b: \mathbb{Z}_b \rightarrow [0,1)$,
given by
$\varphi_b(\sum_{j\ge 0} z_j\, b^j) = \sum_{j\ge 0} z_j\, b^{-j-1}\pmod{1}$,
will be called the {\em $b$-adic Monna map}.
\end{definition}

The restriction of $\varphi_b$ to $\mathbb{N}_0$ is often called the {\em radical-inverse function}
in base $b$.
The Monna map is surjective, but not injective.
It may be inverted in the following sense.
\begin{definition}
We define the {\em pseudoinverse} $\varphi^+_b$ of the $b$-adic Monna map $\varphi_b$  by
\[
\varphi^+_b: [0,1) \rightarrow \mathbb{Z}_b, \quad
\varphi^+_b(\sum_{j\ge 0} x_j\, b^{-j-1}) = \sum_{j\ge 0} x_j\, b^{j}\;,
\]
where $\sum_{j\ge 0} x_j\, b^{-j-1}$ stands for the regular $b$-adic
representation of the element $x\in [0,1)$.
\end{definition}

The image of  $[0,1)$ under $\varphi^+_b$ is the set  $\mathbb{Z}_b\setminus (-\mathbb{N})$.
Furthermore, $\varphi_b\circ \varphi^+_b$ is the identity map on $[0,1)$,
and $\varphi^+_b\circ \varphi_b$ the identity on $\mathbb{N}_0 \subset \mathbb{Z}_b$.
In general,
$z\neq \varphi^+_b(\varphi_b(z))$,  for $z\in \mathbb{Z}_b$.
For example,
if $z=-1$,
then $\varphi^+_b(\varphi_b(-1))=\varphi^+_b(0)=0\neq -1$.

It has been shown in \cite{Hel11a} that the dual group
 $\hat{\mathbb{Z}}_b$
can be written in the form
$\hat{\mathbb{Z}}_b= \{\chi_k: k\in \mathbb{N}_0 \},$
where
$\chi_k: \mathbb{Z}_b\rightarrow \{ c\in \mathbb{C}: |c|=1 \}$,
$\chi_k( \sum_{j\ge 0} z_j b^j ) =     e(\varphi_b(k)(z_0+z_1b+\cdots))$.
We note that $\chi_k$ depends only on a finite number
of digits of $z$ and, hence, this function is well defined.

As in \cite{Hel10a},
we employ the function $\varphi_b^+$ to lift the characters $\chi_k$ to the torus.

\begin{definition}
For $k\in \N_0$,
let
$\gamma_k: [0,1) \rightarrow \{ c\in \mathbb{C}: |c|=1 \}$,
$\gamma_k(x)=     \chi_k(\varphi^+_b(x))$,
denote the $k$th $b$-adic function.
We put $\Gamma_b=\{\gamma_k:  k\in \mathbb{N}_0  \}$
and call it the $b$-adic function system on $[0,1)$.
\end{definition}

There is an obvious generalization of the preceding notions to the
higher-dimensional case.
Let $\mathbf{b}=(b_1, \ldots, b_s)$ be a vector of
not necessarily distinct integers $b_i\ge 2$,
let ${\mathbf{x}}= (x_1, \ldots, x_s)\in [0,1)^s$,
let $\mathbf{z}=(z_1, \ldots, z_s)$ denote an element of the compact product group
$\mathbb{Z}_\mathbf{b}= \mathbb{Z}_{b_1}\times \cdots \times \mathbb{Z}_{b_s}$
of $\mathbf{b}$-adic integers,
and let ${\mathbf{k}}= (k_1, \ldots, k_s)\in \mathbb{N}_0^s$.
We define
$\varphi_{\mathbf{b}}(\mathbf{z}) = (\varphi_{b_1}(z_1), \ldots, \varphi_{b_s}(z_s))$,
and
$\varphi^+_{\mathbf{b}}(\mathbf{x}) = (\varphi^+_{b_1}(x_1), \ldots, \varphi^+_{b_s}(x_s))$.

Let
$\chi_{\mathbf{k}}(\mathbf{z}) = \prod_{i=1}^s \chi_{k_i}(z_i)$,
where $\chi_{k_i} \in \hat{\mathbb{Z}}_{b_i}$,
and define
$\gamma_{\mathbf{k}}({\mathbf{x}}) = \prod_{i=1}^s \gamma_{k_i}(x_i)$,
where $\gamma_{k_i} \in \Gamma_{b_i}$, $1\le i\le s$.
Then
$\gamma_{\mathbf{k}} = \chi_{\mathbf{k}}\circ \varphi^+_{\mathbf{b}}$.
Let $\Gamma_{\mathbf{b}}^{(s)} = \{ \gamma_{\mathbf{k}}:
\mathbf{k} \in \mathbb{N}_0^s \}$
denote the {\em $\mathbf{b}$-adic function system} in dimension $s$.

The dual group $\hat{\Z}_\bfb$ is an orthonormal basis of
the Hilbert space $L^2(\Z_\bfb)$.
A rather elementary proof of this result is given in \cite[Theorem 2.12]{Hel11a}.

For the Walsh functions defined below,
we refer the reader to \cite{Dick10a,Hel93a,Hel98a} for
elementary properties of these functions and to
\cite{Sch90a} for the background in harmonic analysis.

\begin{definition}
For $k\in \mathbb{N}_0$,
$k= \sum_{j\ge 0} k_j b^j$,
and $x\in [0,1)$, with regular $b$-adic representation
$x= \sum_{j\ge 0} x_j b^{-j-1}$,
the {\em $k$th Walsh function in base $b$} is defined by
$w_k(x) = e( (\sum_{j\ge 0} k_j x_j)/b)$.
For ${\mathbf{k}}\in \mathbb{N}_0^s$, ${\mathbf{k}}=(k_1, \ldots, k_s)$,
and ${\mathbf{x}}\in [0,1)^s$,
${\mathbf{x}}=(x_1, \ldots, x_s)$,
we define the {\em ${\mathbf{k}}$th Walsh function $w_{\mathbf{k}}$
in base $\mathbf{b}=(b_1, \ldots, b_s)$} on  $[0,1)^s$ as the following product:
$w_{\mathbf{k}}(\mathbf{x}) = \prod_{i=1}^s
	w_{k_i}(x_i)$,
where $w_{k_i}$ denotes the $k_i$th Walsh function in base $b_i$,
$1\le i\le s$.
The Walsh function system in base $\mathbf{b}$
in dimension $s$ is denoted by
$\mathcal{W}_{\mathbf{b}}^{(s)} = \{w_{\mathbf{k}}:   \mathbf{k}\in \mathbb{N}_0^s\}$.
\end{definition}

The trigonometric function system defined below
is the classical function system in the theory of uniform distribution of sequences
(see the monograph \cite{Kui74a}).

\begin{definition}
Let $k\in \mathbb{Z}$. The $k$th {\em trigonometric function}
$e_k$ is defined as $ e_k: [0,1) \rightarrow \mathbb{C}$, $ e_k(x)
= e(kx)$. For $\bfk=(k_1, \ldots, k_s)\in \mathbb{Z}^s$, the {\em
$\bfk$th trigonometric function} $e_{\bfk}$ is defined as
$e_{\bfk}: [0,1)^s \rightarrow \mathbb{C}$, $e_{\bfk}({\bfx}) =
\prod_{i=1}^s e(k_ix_i)$, $\bfx=(x_1, \ldots, x_s)\in [0,1)^s$.
The {\em trigonometric function system} in dimension $s$ is
denoted by $\mathcal{T}^{(s)} = \{e_{\bfk}: \bfk\in
\mathbb{Z}^s\}$.
\end{definition}

The following presentation complements the concepts discussed in \cite{Hel10c}.
As will become clear,
any finite number of factors can be accomodated.
For given dimensions $s_1, s_2$ and $s_3$,
with $s_i \in \mathbb{N}_0$, not all equal to 0,
put $s=s_1+s_2+s_3$ and write a point $\bfy\in \R^s$ in the form $\bfy=(\bfy^{(1)}, \bfy^{(2)}, \bfy^{(3)})$
with components $\bfy^{(j)}\in \R^{s_j}$, $j=1,2,3$.
Let us fix two vectors of  bases $\mathbf{b}^{(1)}=(b_1, \ldots, b_{s_1})$, and
$\mathbf{b}^{(2)}=(b_{s_1+1}, \ldots, b_{s_1+s_2})$
with not necessarily distinct integers $b_i\ge 2$.
Let $\mathbf{k} = (\mathbf{k}^{(1)}, \mathbf{k}^{(2)}, \mathbf{k}^{(3)})$,
with components $\mathbf{k}^{(1)}\in \mathbb{N}_0^{s_1}$,
$\mathbf{k}^{(2)}\in \mathbb{N}_0^{s_2}$, and $\mathbf{k}^{(3)}\in \Z^{s_3}$.
The tensor product
$\xi_{\mathbf{k}}= w_{\mathbf{k}^{(1)}} \otimes \gamma_{\mathbf{k}^{(2)}} \otimes e_{\mathbf{k}^{(3)}}$,
where
$w_{\mathbf{k}^{(1)}}\in \mathcal{W}_{\mathbf{b}^{(1)}}^{(s_1)}$,
$\gamma_{\mathbf{k}^{(2)}}\in \Gamma_{\mathbf{b}^{(2)}}^{(s_2)}$,
and $e_{\mathbf{k}^{(3)}}\in \mathcal{T}^{(s_3)}$,
defines a function $\xi_{\mathbf{k}}$ on the $s$-dimensional unit cube,
\[
\xi_{\mathbf{k}}: [0,1)^s \rightarrow \mathbb{C},
\quad \xi_{\mathbf{k}}(\mathbf{x}) =
w_{\mathbf{k}^{(1)}}(\mathbf{x}^{(1)})  \gamma_{\mathbf{k}^{(2)}}(\mathbf{x}^{(2)})
    e_{\mathbf{k}^{(3)}}(\mathbf{x}^{(3)})\;,
\]
where $\mathbf{x}= (\mathbf{x}^{(1)}, \mathbf{x}^{(2)}, \mathbf{x}^{(3)})\in [0,1)^s$.

\begin{definition}
Let $s_1, s_2, s_3\in \N_0$, not all $s_i$ equal to 0, and put
$s=s_1+s_2+s_3$.
The family of functions
\[
\mathcal{W}_{\mathbf{b}^{(1)}}^{(s_1)}\otimes \Gamma_{\bfb^{(2)}}^{(s_2)} \otimes \mathcal{T}^{(s_3)}=
    \{ \xi_\bfk,
    \bfk= (\bfk^{(1)},\bfk^{(2)}, \bfk^{(3)})\in \N_0^{s_1}\times \N_0^{s_2}\times \Z^{s_3}\},
\]
is called a {\em hybrid function system} on $[0,1)^s$.
\end{definition}

\begin{remark}
It follows from \cite[Theorem 1 and Corollary 4]{Hel10c} and the techniques exhibited in
\cite{Hel11a}
that such hybrid function systems are an orthonormal basis of $L^2([0,1)^s)$.
\end{remark}

\section{The hybrid spectral test}\label{s:spectral}

\begin{remark}
All of the following results remain valid if we change the order of the factors
in the hybrid function system,
as it will become apparent from the proofs below.
In particular, out of the given $s$ coordinates,
we may select arbitrary $s_1$ coordinates and assign to them the Walsh system
$\mathcal{W}_{\mathbf{b}^{(1)}}^{(s_1)}$ in some base $\bfb^{(1)}$,
treat $s_2$ of the remaining $s -s_1$ coordinates with a $\bfb^{(2)}$-adic system $\Gamma_{\bfb^{(2)}}^{(s_2)}$,
and use the system $\mathcal{T}^{(s_3)}$ for the final $s_3$ coordinates.
\end{remark}

\begin{definition}
Let $s\in \N$.
By an  \emph{$s$-dimensional index set} $\Lambda$ we understand one of the
additive semigroups $(\Z^s, +)$, $(\N_0^s, +)$, and $(\N^s, +)$,
or finite direct products of these semigroups
such that  the dimensions of the factors add up to $s$.

Let $\Lambda^*$ denote the index set $\Lambda\setminus \{\mathbf{0}\}$.
\end{definition}

Examples of $s$-dimensional index sets are direct products of the form
$\N_0^{s_1}\times \N_0^{s_2}\times \Z^{s_3}$,
where $s=s_1+s_2+s_3$,
with $s_1, s_2, s_3\in \N_0$, not all $s_i$ equal to 0,
as they appear in hybrid function systems.

If $\omega = ({\bf x}_n)_{n\ge 0}$ is
a -possibly finite- sequence in $[0,1)^s$ with at least $N$ elements,
and if $f:\, [0,1)^s$ $\to$ ${\C}$,
we define
\[
S_N(f,\omega) = \frac{1}{N} \sum_{n=0}^{N-1} f({\bf x}_n).
\]

\begin{definition}
Let $\Lambda$ be an $s$-dimensional index set.
A subclass $\myF=\{ \xi_\bfk: \bfk\in \Lambda  \}$
of the class of Riemann integrable functions on $[0,1)^s$
is called \emph{a uniform distribution determining (u.d.d.) function system on $[0,1)^s$}
if the functions $\xi_\bfk$ are normalized in the sense that
$||\xi_\bfk||_\infty = \sup\{|\xi_\bfk(\bfx)|: \bfx\in [0,1)^s\}\le 1$ for all $\bfk$,
and $\int_{[0,1)^s} \xi_\bfk d\lambda_s = 0$ for all $\bfk\in \Lambda^*$,
and if, for any sequence $\omega$ in $[0,1)^s$,
the property
\[
\forall \bfk\in \Lambda^*: \lim_{N\to \infty} S_N(\xi_\bfk, \omega) =0
\]
implies the uniform distribution of $\omega$ in $[0,1)^s$.
\end{definition}

Examples of u.d.d. classes of functions on $[0,1)^s$ are
the hybrid function systems
$\mathcal{W}_{\mathbf{b}^{(1)}}^{(s_1)}\otimes \Gamma_{\bfb^{(2)}}^{(s_2)} \otimes \mathcal{T}^{(s_3)}$.
This follows from the hybrid Weyl Criterion \cite[Theorem 1]{Hel10c}, whose proof is easily adapted
to provide for the non-prime bases $b_i$ we allow here in the factor $\Gamma_{\bfb^{(2)}}^{(s_2)}$.

\begin{definition}\label{def:weightfunction}
Let $||\cdot||$ be an arbitrary norm on $\R^s$ and let $\Lambda$ be an $s$-dimensional index set.
We call a real-valued function $\rho$ a \emph{weight function} on $\Lambda$
if, for all $\bfk\in \Lambda$, $\rho(\mathbf{k}) > 0$, and if
for all $\epsilon>0$, there exists a positive real number $K_0=K_0(\epsilon)$ such that
$\rho(\bfk)<\epsilon$ for all $\bfk\in\Lambda$ with $||\bfk||>K_0$.

\end{definition}

With a u.d.d. function system $\myF=\{ \xi_\bfk: \bfk\in \Lambda  \}$ on $[0,1)^s$
and a weight function $\rho$ on $\Lambda$ we may associate the array of weighted functions
\begin{equation*}\label{s:spectral:eqn:array}
 \left( \rho(\bfk) \xi_\bfk \right)_{\bfk\in \Lambda}.
\end{equation*}

The operator $S_N(\cdot, \omega)$ is linear, hence we may write
\[
S_N( \left( \rho(\bfk) \xi_\bfk \right)_{\bfk\in \Lambda}, \omega)
    = \left( \rho(\bfk) S_N(\xi_\bfk, \omega) \right)_{\bfk\in \Lambda}.
\]

\begin{definition}
\label{definition:generalspectraltest}
Let $\myF=\{ \xi_\bfk: \bfk\in \Lambda  \}$ be a u.d.d. function system on $[0,1)^s$, and let
$\rho$ be a weight function on $\Lambda$.
For a given sequence $\omega = (\mathbf{x}_n)_{n\ge 0}$ in $[0,1)^s$,
the \emph{spectral test} $\sigma_N(\omega)$
of the first $N$ elements of $\omega$,
with respect to $\myF$ and $\rho$,
is defined as
\begin{align*}
\sigma_N(\omega) &= || \left( \rho(\bfk)  \right)_{\bfk\in \Lambda^*}||_\infty^{-1}\
    || \left( \rho(\bfk) S_N(\xi_\bfk, \omega) \right)_{\bfk\in \Lambda^*} ||_\infty \\
    &= \sup_{\bfk\in \Lambda^*} \{ \rho(\bfk) \}^{-1}
    \sup_{\bfk\in \Lambda^*} \left\{
        \rho(\bfk) \left| S_N( \xi_\mathbf{k}, \omega) \right| \right\}.
\end{align*}
Let $\alpha>1$ be a given real number.
If the weight function $\rho$ fulfills the additional condition
\[
\sum_{\bfk \in \Lambda} \rho(\bfk)^\alpha < \infty,
\]
then the $L^\alpha$-\emph{diaphony} $F_N^{(\alpha)}(\omega)$
of the first $N$ elements of $\omega$,
with respect to $\myF$ and $\rho$,
is defined as
\begin{align*}
F_N^{(\alpha)}(\omega) &= ||  \left( \rho(\bfk) \right)_{\bfk\in \Lambda^*}||_\alpha^{-1}
    \ || \left( \rho(\bfk) S_N(\xi_\bfk, \omega) \right)_{\bfk\in \Lambda^*} ||_\alpha\\
   &= \left( \sum_{\bfk\in \Lambda^*} \rho(\bfk)^\alpha\right)^{-1/\alpha}\
    \left( \sum_{\bfk\in \Lambda^*} \rho(\bfk)^\alpha \left| S_N( \xi_\mathbf{k}, \omega) \right|^\alpha\right)^{1/\alpha}.
\end{align*}
\end{definition}

We note that the fact $|S_N(\xi_\bfk, \omega)|\le 1$ implies that the spectral test as well as
diaphony are normalized: $0\le \sigma_N(\omega)\le 1$, and  $0\le F_N^{(\alpha)}(\omega)\le 1$.

\begin{theorem}
\label{thm:spectraltest}
Let $\myF$, $\rho$ and $\sigma_N(\omega)$ be as in Definition \ref{definition:generalspectraltest}.
Then
\begin{enumerate}
\item\label{item:sigmaN-1}
The quantity $\sigma_N(\omega)$ is a maximum.
\item\label{item:sigmaN-2}
The sequence $\omega$ is uniformly distributed modulo one if and only if
\[
\lim_{N\to \infty} \sigma_N(\omega)= 0.
\]
\end{enumerate}

\end{theorem}

\begin{proof}
The proof generalizes the arguments in \cite[Sec. 5.2]{Hel98a}.
For an arbitrary positive integer $K$,
let
\[
A_{N,K}= \sup \left\{
       \rho(\bfk) \left| S_N( \xi_\mathbf{k}, \omega) \right|:
        0<||\mathbf{k}|| \le K  \right\},
\]
and
\[
B_{N,K}= \sup \left\{
        \rho(\bfk) \left| S_N( \xi_\mathbf{k}, \omega) \right|:
        ||\mathbf{k}|| > K   \right\}.
\]
Clearly,
\begin{equation}\label{eqn:sigmaN}
\sigma_N(\omega) = \max \left\{ A_{N,K}, B_{N,K}\right\}.
\end{equation}
We have $\sigma_N(\omega)> 0$.
Otherwise, all terms $S_N(\xi_\bfk, \omega)$,
$\bfk\in \Lambda^*$, would be equal to zero,
which is impossible.
Hence,  there exists $\delta > 0$ such that
$\delta < \sigma_N(\omega)$,
and an index $K_0 = K_0(\delta)\in \mathbb{N}$ such that
for all $\mathbf{k}$ with $	||\mathbf{k}||> K_0$ we have $\rho(\bfk) < \delta$.
This implies
\[
B_{N,K_0}
\le
\sup \left\{
        \rho(\mathbf{k}): ||\mathbf{k}|| > K_0 \right\}
\le \delta < \sigma_N(\omega).
\]
Hence, $\sigma_N(\omega) = A_{N,K_0}$.
We observe that the set $\{\mathbf{k}: ||\mathbf{k}|| \le K_0  \}$ is finite.
In this context,
we recall that all norms on $\R^s$ are equivalent.
This proves (\ref{item:sigmaN-1}).

In order to prove (\ref{item:sigmaN-2}),
suppose first that $\lim_{N\to \infty} \sigma_N(\omega) = 0$.
This implies that $\lim_{N\to \infty}S_N(\xi_\mathbf{k}, \omega) = 0$
for all $\mathbf{k} \in \Lambda^*$.
The class $\myF$ is u.d.d.,
hence $\omega$ is uniformly distributed in $[0,1)^s$.

To prove the converse,
assume that $\omega$ is uniformly distributed in $[0,1)^s$.
For any $\epsilon > 0$, there exists a positive integer
$K_0 = K_0(\epsilon)$ such that $\rho(\bfk) < \epsilon$
for all $\mathbf{k}$ with $||\mathbf{k}|| > K_0$.
As in the proof of part (\ref{item:sigmaN-1}),
this gives $B_{N,K_0}\le \epsilon$, and, due to (\ref{eqn:sigmaN}),
\[
\sigma_N(\omega) \le A_{N,K_0} + \epsilon.
\]
The number $A_{N,K_0}$ is a maximum and the class $\myF$ is u.d.d..
This implies the existence of $N_0=N_0(\epsilon)\in \N$ with the property
\[
\forall\ N \ge N_0(\epsilon):
    \ A_{N,K_0} = \max \left\{
        \rho(\bfk) \left| S_N( \xi_\mathbf{k}, \omega) \right|:
        0< ||\mathbf{k}|| \le K_0
        \right\}
< \epsilon.
\]

We deduce the relation
$$
\forall\ N \ge N_0(\epsilon): \ \sigma_N(\omega) < 2 \epsilon.
$$
This proves the theorem.
\end{proof}

\begin{corollary}
Let $\myF$, $\rho$ and $\sigma_N(\omega)$ be as in Definition \ref{definition:generalspectraltest}
and let $K$ denote an arbitrary positive integer.
Then we have the following inequality of \etk \ for the spectral test:
\begin{equation*}
\sigma_N(\omega) \le \max \Big\{
    \max_{0<||\mathbf{k}|| \le K} \left\{
        \rho(\bfk) \left| S_N( \xi_\mathbf{k}, \omega) \right| \right\},
        \sup_{||\bfk||>K}\left\{\rho(K)\right\}
        \Big\}.
\end{equation*}
\end{corollary}

\begin{proof}
This follows directly from (\ref{eqn:sigmaN}).
\end{proof}

\begin{theorem}
\label{thm:diaphony}
Let $\myF$, $\rho$ and $F_N^{(\alpha)}(\omega)$ be as in Definition \ref{definition:generalspectraltest}
and suppose that
\[
\sum_{\bfk\in \Lambda^*} \rho(\bfk)^\alpha< \infty.
\]
Then sequence $\omega$ is uniformly distributed in $[0,1)^s$ if and only if
\[
\lim_{N\to \infty} F_N^{(\alpha)}(\omega)= 0.
\]
\end{theorem}

\begin{proof}
We adapt the proof of \cite[Theorem 2]{Hel10c}
and the splitting technique used in the proof of Theorem \ref{thm:spectraltest} to the case of diaphony.
\end{proof}

%
%

\section{Examples}\label{s:examples}
\subsection{Examples I: integration lattices}
\label{ss:spectraltest_examples}

Definition \ref{definition:generalspectraltest} generalizes various known notions of
the spectral test and of diaphony.

%

The spectral test has its origin in pseudorandom number generation.%
It measures the ``coarseness'' of  lattices
that can be associated with certain linear types of generators.
We refer to \cite[Ch. 3.3.4.]{Knu98a}) for a seminal discussion,
and to the surveys \cite{Lec97a,Nie92a} as well as to \cite{Hel98f}.

We recall some notions from the theory of lattices (for details see \cite{Nie92a,Slo94a}).
An $s$-dimensional \emph{lattice} is a discrete subset of $\R^s$ which is closed under addition and subtraction.
For any $s$-dimensional lattice $L$, there exists a lattice basis,
by which we understand $s$ independent vectors
$\bfg_1, \ldots, \bfg_s$  in $\R^s$
such that
\begin{equation*}
L= \left\{ \sum_{i=1}^s t_i \bfg_i: \ t_i\in \Z  \right\}.
\end{equation*}
Lattice bases are not unique.

A point $\bfg\in L\setminus \{\bfzero\}$ is called a \emph{primitive point} of $L$ if the line segment joining the origin $\bfzero$ and $\bfg$
does not contain any other point of $L$.

The \emph{dual lattice} $L^\bot$ of $L$ is defined as
\begin{equation*}
L^\bot = \left\{\bfz\in \R^s:\ \bfz \cdot \bfx \in \Z, \text{ for all } \bfx\in L  \right\},
\end{equation*}
where $\bfz \cdot \bfx$ denotes the usual inner product.

An $s$-dimensional \emph{integration lattice} is an $s$-dimensional lattice that contains $\Z^s$ as
a sublattice.
An $s$-dimensional $N$-point \emph{lattice rule} is an $s$-dimensional integration lattice $L$
of the form
\begin{equation}\label{eqn:integrationlattice}
L = \bigcup_{n=0}^{N-1} \left( \bfx_n + \Z^s  \right),
\end{equation}
where $\bfx_0, \ldots, \bfx_{N-1}$ are the $N$ distinct points of $L$ that belong to $[0,1)^s$.

The spectral test for lattices  is defined as follows.
\begin{definition}\label{definition:classicalspectraltest}
Let $L$ be an $s$-dimensional lattice in $\mathbb{R}^s$.
We will call a  family $\hyperfamily$ of parallel hyperplanes
$\hyper{c}$, $c \in C$,
in  $\mathbb{R}^s$ \emph{a cover of $L$} if
(i) $L \subseteq \bigcup_{c\in C} \hyper{c}$, and
(ii) $C$ is the smallest set (in the sense of set-inclusion)
with this property.

Let $\hyperfamily$ denote an arbitrary cover of $L$ in $\mathbb{R}^s$.
The \emph{spacing} $d(\hyperfamily)$ of $\hyperfamily$ will denote the minimal
distance between adjacent hyperplanes in this family,
$
d(\hyperfamily) = \inf \left\{  d(\hyper{c}, \hyper{d}) : c\neq d, \
        c,d  \in C \right\}.
$

We define \emph{the spectral test} of $L$ as the number
\[
\sigma (L) := \sup \left\{ d(\hyperfamily): \hyperfamily \text{ is a
        cover of } L \right\}.
\]
\end{definition}

The following theorem is well known in the theory of pseudorandom number generation
(see \cite{Fis96a,Hel98a,Knu98a,Lee95c}).
Computational aspects are discussed in \cite{Lec99a}.

\begin{theorem}
Let $L$ be an $s$-dimensional lattice.
Then
\begin{equation*}\label{eqn:classicalsigma}
\sigma(L) = 1/\min \{||\bfg||_2: \bfg\in L^\bot, \bfg \text{ primitive }  \},
\end{equation*}
where $||\bfg||_2 = (g_1+ \dots + g_s)^{1/2}$ is the Euclidean norm on $\R^s$.
\end{theorem}


In the field of quasi-Monte Carlo integration,
similar concepts have been developed within the context of
{\em good lattice points}.

\begin{definition}\label{def:GLP}
For a given integer $N\ge 2$ and for a given dimension $s\ge2$,
we call an integer point
$\mathbf{a} = (a_1, \ldots, a_s)\in$ $\mathbb{Z}^s$,
$\gcd(a_i, N)=1$, $1\le i\le s$,
a \emph{good lattice point} modulo $N$ if the finite sequence
\begin{equation}\label{eqn:GLPa}
\omega_\mathbf{a} = \left( \left\{ \frac{n}{N}\;
        \mathbf{a} \right\} \right)_{n=0}^ {N-1}
\end{equation}
is ``very uniformly'' distributed in $[0,1)^s$
in the sense that its discrepancy is low.
\end{definition}

The conditions $\gcd(a_i,N)=1$, $1\le i\le s$, ensure
that all points in $\omega_\bfa$ and in its projections to lower dimensions are distinct.

We may view $\omega_\bfa$ as the node set of
an $s$-dimensional $N$-point lattice rule.
Let $L(\omega_\bfa)$ denote the integration lattice associated with $\omega_\bfa$
defined by (\ref{eqn:integrationlattice}).
The dual lattice is given by
$L(\omega_\bfa)^\bot = \left\{ \bfk\in\Z^s: \bfk\cdot \bfa \equiv 0 \pmod{N}  \right\}$.

With an integer point $\bfa\in \Z^s$ subject to the conditions of Definition \ref{def:GLP},
we may associate several figures of merit.
In view of the spectral test for lattices,
define $\sigma(\bfa,N)=\sigma (L(\omega_\bfa))$.
Then
\begin{equation*}\label{eqn:classicalspectraltest2}
\sigma(\bfa,N)= 1/
        \min \left\{ ||\mathbf{k}||_2: \mathbf{k} \in
        \mathbb{Z}^s\setminus \{ \mathbf{0} \},\
        \mathbf{k}\cdot \mathbf{a} \equiv 0 \pmod{N}  \right\}.
\end{equation*}

For a real number $\alpha> 1$,
define
\begin{equation*}\label{eqn:Palpha}
P_\alpha(\mathbf{a},N) = \sum_{\scriptstyle
\mathbf{k} \neq \mathbf{0}
\atop \scriptstyle \mathbf{k}\cdot \mathbf{a} \equiv 0 \pmod{N}}
\frac{1}{r(\mathbf{k})^\alpha},
\end{equation*}
where summation is over integer vectors $\bfk$,
and
\[
r(\mathbf{k})= \prod_{i=1}^s \max\{1, |k_i| \},\quad
\mathbf{k} = (k_1, \ldots, k_s) \in \mathbb{Z}^s.
\]

Another important figure of merit is the
\emph{Babenko-Zaremba index}
\begin{equation*}\label{eqn:BZindex}
\varkappa (\mathbf{a},N) = 1 / \min \left\{ r(\mathbf{k}): \mathbf{k} \in
        \mathbb{Z}^s\setminus \{ \mathbf{0} \},\
         \mathbf{k}\cdot \mathbf{a} \equiv 0 \pmod{N}  \right\},
\end{equation*}
see the monographs  \cite{Nie92a,Slo94a}.

The three quantities $\sigma(\bfa,N)$, $P_\alpha(\bfa,N)$, and $\varkappa(\bfa,N)$
are special cases of the spectral test introduced in Definition \ref{definition:generalspectraltest}
applied to the sequence $\omega_\bfa$.
In order to establish this connection,
we employ the following result of Sloan and Kachoyan~\cite{Sloan87a}
(for the proof see \cite[Lemma 5.21]{Nie92a} and \cite[Lemma 2.7]{Slo94a}):

\begin{lemma}
Let $\omega = (\bfx_n)_{n=0}^{N-1}$ be the sequence of nodes of an $s$-dimensional
$N$-point lattice rule $L$.
Then
\[
S_N(e_\bfk, \omega) =
  \begin{cases}
   1, & \text{if } \,\bfk\in L^\bot, \\
   0,  & \text{if } \,\bfk\not\in L^\bot.
  \end{cases}
\]
\end{lemma}

\begin{example}
Let $\mathcal{F} = \mathcal{T}^{(s)}$
(hence $\Lambda=\Z^s$), and put $\omega_\bfa = (\{(n/N)\bfa\})_{n=0}^{N-1}$,
where $\bfa\in \Z^s$ is subject to the conditions in Definition \ref{def:GLP}.
Then,
\begin{enumerate}
\item for the choice $\rho(\mathbf{k}) = ||\mathbf{k}||_2^{-1}$ for $\bfk\neq \bfzero$,
    the hybrid spectral test $\sigma_N(\omega_\bfa)$ introduced in Definition \ref{definition:generalspectraltest},
    with respect to $\myF$ and $\rho$, is equal to the classical spectral test
    of Definition
    \ref{definition:classicalspectraltest} applied to the integration lattice
    $L(\omega_\bfa)$.

\item for $\rho(\mathbf{k}) = r(\bfk)^{-1}$ for $\bfk\neq \bfzero$,
we obtain $\sigma_N(\omega_\bfa) = \varkappa(\bfa,N)$,
the Babenko-Zaremba index defined in (\ref{eqn:BZindex}).

\item for $\rho(\mathbf{k}) = r(\bfk)^{-1}$ for $\bfk\neq \bfzero$, and $\alpha>1$,
we get
\[
\left( \sum_{\bfk\neq \bfzero} r(\bfk)^{-\alpha} \right)^{1/\alpha} F^{(\alpha)}_N (\omega_\bfa)
    = \left( P_\alpha(\bfa,N) \right)^{1/\alpha},
\]
for the $L^\alpha$-diaphony of the sequence $\omega_\bfa$.
\end{enumerate}
\end{example}

\begin{example}
By obvious  choices of $\myF$ and $\rho$ and with $\alpha=2$,
we obtain the classical diaphony of  Zinterhof \cite{Zin76a},
the dyadic (Walsh) diaphony of Hellekalek and Leeb \cite{Hel96b},
the $b$-adic (Walsh) diaphony versions of Grozdanov et al. \cite{MR2246893,MR1980679,MR1829550},
the $p$-adic diaphony of Hellekalek \cite{Hel10a},
and the more general notion of {\em hybrid} diaphony that was introduced in \cite{Hel10c}.
\end{example}

\begin{example}
By obvious choices for $\myF$ and $\rho$,
we obtain the Walsh spectral test of Hellekalek~\cite{Hel02a}.
\end{example}

\subsection{Examples II: extreme and star discrepancy}\label{s:examples2}

The hybrid spectral test introduced in Definition \ref{definition:generalspectraltest} does not include the extreme
discrepancy and the star discrepancy, but we may approximate these two measures of
uniform distribution arbitrarily close by suitable versions of the hybrid spectral test.

We recall the definition of discrepancy.
Let $\mathcal{J}$ denote the class of all subintervals of $[0,1)^s$
of the form $\prod_{i=1}^s [u_i,v_i)$, $0\le u_i<v_i\le 1$, $1\le i\le s$,
and let $\mathcal J^*$ denote the subclass of $\mathcal{J}$ of intervals
of the type $\prod_{i=1}^s [0, v_i)$ anchored at the origin.
The extreme discrepancy and the star discrepancy of a sequence
are defined as follows (see \cite{Kui74a,Nie92a}).

\begin{definition}
Let $\omega = (\mathbf{x}_n)_{n\ge 0}$ be a sequence in $[0,1)^s$.
The {\em (extreme) discrepancy}
$D_N(\omega)$ of the first $N$ elements of $\omega$ is defined as
\[
D_N(\omega) \;=\; \sup_{J\in \mathcal{J}} \left| S_N(\myone_J- \lambda_s(J), \omega)
    \right|.
\]
The {\em star discrepancy}
$D_N^*(\omega)$ of the first $N$ elements of $\omega$ is defined as
\[
D_N^*(\omega) \;=\; \sup_{J\in \mathcal{J}^*} \left| S_N(\myone_J - \lambda_s(J), \omega)
    \right|.
\]
\end{definition}

We first approximate $D_N(\omega)$ and $D_N^*(\omega)$ by discrete discrepancies.


\begin{definition}
Let $\mathbf{b}=(b_1, \ldots, b_s)$,
with not necessarily distinct integers $b_i\ge 2$.
A {\em $\mathbf{b}$-adic  interval} in the resolution class defined by $\mathbf{g}=(g_1, \ldots, g_s)\in \N_0^s$
(or with resolution $\mathbf{g}$)
is a sub\-interval  of $[0,1)^s$ of the form
\[
\prod_{i=1}^s \left[
    a_i b_i^{-g_i}, d_i b_i^{-g_i}
	\right),
    \ 0\le a_i < d_i \le b_i^{g_i}, \;  a_i, d_i \in \mathbb{N}_0, \;1\le i\le s\; .
\]
We denote the class of all $\bfb$-adic intervals with resolution $\bfg$ by $\mathcal{J}_{\bfb,\bfg}$.
The subclass of those $\bfb$-adic intervals anchored at the origin
will be denoted by $\mathcal{J}_{\bfb,\bfg}^*$.
Further, let
\[
\mathcal{J}_\bfb = \bigcup_{\bfg\in \N_0^s} \mathcal{J}_{\bfb,\bfg}
\]
denote the class of all $\bfb$-adic intervals in $[0,1)^s$ and put
\[
\mathcal{J}_\bfb^* = \bigcup_{\bfg\in \N_0^s} \mathcal{J}_{\bfb,\bfg}^*.
\]
\end{definition}

For a given resolution
$\mathbf{g}\in \mathbb{N}_0^s$,
we define the  domains
\begin{align*}
\Delta_{\mathbf{b}}(\mathbf{g}) &= \left\{\mathbf{k}=(k_1, \ldots, k_s)\in \mathbb{N}_0^s :\
0\le k_i < b_i^{g_i}, 1\le i\le s  \right\},\\
\Delta_{\mathbf{b}}^*(\mathbf{g}) &= \Delta_{\mathbf{b}}(\mathbf{g})
    \setminus \{\mathbf{0} \}\;.\nonumber
\end{align*}
and
\[
\nabla_{\mathbf{b}}(\mathbf{g}) = \left\{\mathbf{k}=(k_1, \ldots, k_s)\in \mathbb{N}_0^s :\
1\le k_i \le b_i^{g_i}, 1\le i\le s  \right\}.
\]
We note that $\Delta_\bfb(\bf0)=\{ \bfzero\}$ and $\nabla_\bfb(\bf0)=\emptyset$.
Further, we observe that we may write the intervals in $\mathcal{J}_{\bfb,\bfg}$ in the form
\[
\mathcal{J}_{\bfb,\bfg} =
\left\{ I_{\bfa,\bfd;\bfg}: (\bfa,\bfd)\in \Delta_\bfb(\bfg)\times \nabla_\bfb(\bfg) \right\},
\]
where $ I_{\bfa,\bfd;\bfg}= \prod_{i=1}^s [\varphi_{b_i}(a_i), \varphi_{b_i}(d_i))$,
and $\bfa=(a_1, \ldots, a_s)$,
and $ \bfd=(d_1, \ldots, d_s)$.
The intervals in $\mathcal{J}_{\bfb,\bfg}^*$ are  of the form $I_{\bfzero,\bfd;\bfg}$,
with $\bfd\in \nabla_{\mathbf{b}}(\mathbf{g})$.

\begin{definition}
Let $\omega = (\mathbf{x}_n)_{n\ge 0}$ be a sequence in $[0,1)^s$,
let $\mathbf{b}=(b_1, \ldots, b_s)$,
with not necessarily distinct integers $b_i\ge 2$, and let $\bfg\in \N_0^s$ be a given resolution vector.
The {\em discrete (extreme) discrepancy} in base $\bfb$, for resolution $\bfg$,
 of the first $N$ elements of $\omega$ is defined as
\[
D_{N;\bfb,\bfg}(\omega) = \max_{I\in \mathcal{J}_{\bfb,\bfg}} \left| S_N(\myone_I- \lambda_s(I), \omega) \right|.
\]
The {\em discrete star discrepancy} in base $\bfb$, for resolution $\bfg$,
of the first $N$ elements of $\omega$ is defined as
\[
D_{N;\bfb,\bfg}^*(\omega) = \max_{I\in \mathcal{J}_{\bfb,\bfg}^*} \left| S_N(\myone_I- \lambda_s(I), \omega) \right|.
\]
\end{definition}

\begin{theorem}\label{theorem:discrepancy1}
Let $\bfb=(b_1, \ldots, b_s)$ be a vector of $s$ not necessarily distinct integers $b_i\ge 2$.
Then, for all $\bfg=(g_1, \ldots, g_s)\in\N^s$,
\begin{align*}
D_{N;\bfb,\bfg}(\omega)     &\le D_N(\omega) \le \epsilon_\bfb(\bfg) + D_{N;\bfb,\bfg}(\omega),\\
D_{N;\bfb,\bfg}^*(\omega) &\le D_N^*(\omega) \le \epsilon^*_\bfb(\bfg) + D_{N;\bfb,\bfg}^*,
\end{align*}
where the error terms $\epsilon_\bfb(\bfg)$ and $\epsilon^*_\bfb(\bfg)$ are given by
\[
\epsilon_\bfb(\bfg) = 1 - \prod_{i=1}^s (1- 2 b_i^{-g_i}),
\quad
\epsilon^*_\bfb(\bfg) = 1 - \prod_{i=1}^s (1-  b_i^{-g_i}).
\]
\end{theorem}

\begin{proof}
The inequalities $D_{N;\bfb,\bfg}(\omega) \le D_N(\omega)$ and $D_{N;\bfb,\bfg}^*(\omega) \le D_N^*(\omega)$
are trivial.

In order to show $D_N(\omega) \le \epsilon_\bfb(\bfg) + D_{N;\bfb,\bfg}(\omega)$,
we proceed as in the proof of Theorem 3.12 in \cite{Hel13a}.
For an arbitrary subinterval $J$ of $[0,1)^s$ we obtain the following
bound (see \cite[Inequality (6)]{Hel13a}):
\begin{equation}\label{eqn:upperbound01}
\left| S_N(\myone_J - \lambda_s(J), \omega) \right| \le
 \epsilon_\bfb(\bfg) +
    \max_{I\in \mathcal{J}_{\bfb,\bfg}}\{ \left| S_N(\myone_I - \lambda_s(I), \omega) \right|  \}.
\end{equation}
This bound is independent of the choice of $J$.
As a consequence,
\[
D_N(\omega) = \sup_{J\in \mathcal{J}} \left| S_N(\myone_J- \lambda_s(J), \omega) \right|
\le \epsilon_\bfb(\bfg) + D_{N;\bfb,\bfg}(\omega).
\]

In the case of the star discrepancy $D_N^*(\omega)$,
the intervals $J$ are anchored at the origin.
It is easy to see from the proof of Theorem 3.12 in \cite{Hel13a} that this fact allows us
to replace the error term $\epsilon_\bfb(\bfg)$ by $\epsilon_\bfb^*(\bfg)$.
This finishes the proof.
\end{proof}

\begin{remark}\label{rem:bound1}
An elementary analytic argument shows that
$\epsilon_\bfb(\bfg)\leq 2s\delta_\bfg$,
and $\epsilon^*_\bfb(\bfg)\leq s\delta_\bfg$,
where $\delta_\bfg=\max_{1\le i\le s} b_i^{-g_i}$.
\end{remark}

\begin{corollary}
We have the following discretization:
\begin{align*}
D_N(\omega) &= \sup_{I\in \mathcal{J}_\bfb} \left| S_N(\myone_I - \lambda_s(I), \omega) \right|,\\
D_N^*(\omega) &= \sup_{I\in \mathcal{J}_\bfb^*} \left| S_N(\myone_I - \lambda_s(I), \omega) \right|.
\end{align*}
\end{corollary}

We recall that a sequence $\omega$ is called uniformly distributed in $[0,1)^s$ if and only if
\begin{equation}\label{definition:udmod1}
\forall J\in \mathcal{J}:\ \lim_{N\to \infty} S_N(\myone_J - \lambda_s(J), \omega) = 0.
\end{equation}

\begin{corollary}\label{cor:udmod1}
It follows from Inequality (\ref{eqn:upperbound01}) that a sequence $\omega$ is uniformly distributed in $[0,1)^s$
if and only if
\[
\forall I\in \mathcal{J_\bfb}:\ \lim_{N\to \infty} S_N(\myone_I - \lambda_s(I), \omega) = 0.
\]
\end{corollary}

The argument to approximate the discrepancies $D_N$ and $D_N^*$ by a hybrid spectral test goes as follows.
Let $\bfb=(b_1, \ldots, b_s)$ be a vector of $s$ not necessarily distinct integers $b_i\ge 2$.
As index set, we choose $\Lambda= \N_0^s\times\N^s$ and
observe that
\[
\Lambda = \bigcup_{\bfg\in \N_0^s} \left( \Delta_\bfb(\bfg)\times \nabla_\bfb(\bfg) \right).
\]
An index point $(\bfa,\bfd)\in \Lambda$ is called \emph{admissible} if there exists $\bfg\in \N_0^s$ such that
$(\bfa, \bfd)\in \Delta_\bfb(\bfg)\times \nabla_\bfb(\bfg)$ and
the interval $I_{\bfa,\bfd;\bfg}$ belongs to $\mathcal{J}_{\bfb,\bfg}$.
The point  $(\bfa,\bfd)\in \Lambda$ is called \emph{non-admissible} otherwise.
The reader should note that different admissible points may produce the same $\bfb$-adic interval.


As the elements of the functions system $\myF=\{ \xi_{(\bfa, \bfd)}: (\bfa, \bfd)\in \Lambda \}$,
we choose the functions $\myone_I - \lambda_s(I)$, with $I\in \mathcal{J}_\bfb$, or the constant
function $0$, subject to the following parametrization.
For an admissible index $(\bfa,\bfd)\in\Lambda$,
let
\[
\xi_{(\bfa,\bfd)} = \myone_{I_{\bfa,\bfd;\bfg}} - \lambda_s(I_{\bfa,\bfd;\bfg} ).
\]
It follows that $\myF$ contains all functions $1_I - \lambda_s(I)$,
with $I\in \mathcal{J}_\bfb$.
If $(\bfa,\bfd)$ is non-admissible,
define  $\xi_{(\bfa,\bfd)}$ to be identically $0$.
Corollary \ref{cor:udmod1} implies that $\myF$ is u.d.d.

For a given $\bfg\in \N^s$,
we define the weight function $\rho_\bfg$ on $\Lambda$ in the following manner.
If $k\in \N_0$,
with $b$-adic representation $k=k_0 + k_1 b + \cdots$,
we put
\[
v_b(k) =
  \begin{cases}
   0, & \text{if } k=0, \\
   1 + \max\{j: k_j \neq 0\}, & \text{if } k\ge 1.
  \end{cases}
\]
For $(\bfa,\bfd) \in \Lambda$,
$\bfa=(a_1, \ldots, a_s)$, $\bfd=(d_1, \ldots, d_s) $,
let
\begin{equation}\label{def:rho-g}
\rho_\bfg\left((\bfa,\bfd)\right) =
    \begin{cases}
 1, & \text{if } (\bfa, \bfd)\in \Delta_\bfb(\bfg)\times \nabla_\bfb(\bfg),\\
  \prod_{i=1}^s b_i^{-(v_{b_i}(a_i)+v_{b_i}(d_i))},      & \text{otherwise.}
  \end{cases}
\end{equation}
Further, we choose the maximum norm on $\R^s$.
Then $\rho_\bfg$ is a weight function in the sense of Definition \ref{def:weightfunction}.

\begin{theorem}
Let $\Lambda$ and $\myF$ be as above.
For every $\epsilon>0$, there exists  an integer vector $\bfg\in\N^s$ such that
for any sequence $\omega$ in $[0,1)^s$,
the spectral test $\sigma_N(\omega)$ of the first $N$ elements of $\omega$,
with respect to $\myF$ and $\rho_\bfg$,
has the property
\[
|\sigma_N(\omega) - D_N(\omega)|< \epsilon.
\]
\end{theorem}

\begin{proof}
Let $\epsilon>0$ be given.
We choose $\bfg\in \N^s$ such that
\[
\max_{1\le i\le s} b_i^{-g_i} < \epsilon / (4s).
\]
For the function system $\myF$ and for the
weight function $\rho_\bfg$ defined in (\ref{def:rho-g}),
we have
\begin{align*}
\sigma_N(\omega) = &\max \left\{
    \max \{|S_N(\xi_{(\bfa,\bfd)}, \omega)|: (\bfa,\bfd)\in
        \Delta_\bfb(\bfg)\times \nabla_\bfb(\bfg) \}, \right.  \\
    & \ \ \qquad \left. \sup\{ \rho_\bfg((\bfa,\bfd))|S_N(\xi_{(\bfa,\bfd)}, \omega)|: (\bfa,\bfd)\not\in
        \Delta_\bfb(\bfg)\times \nabla_\bfb(\bfg) \} \right\}.
\end{align*}
We have
\[
D_{N;\bfb,\bfg}(\omega) =
\max \left\{|S_N(\xi_{(\bfa,\bfd)}, \omega)|: (\bfa,\bfd)\in
\Delta_\bfb(\bfg)\times \nabla_\bfb(\bfg) \right\},
\]
and
\[
\sup\left\{ \rho_\bfg((\bfa,\bfd))|S_N(\xi_{(\bfa,\bfd)}, \omega)|: (\bfa,\bfd)\not\in
        \Delta_\bfb(\bfg)\times \nabla_\bfb(\bfg) \right\} \le
        \max_{1\le i\le s} b_i^{-1-g_i}.
\]
The choice of $\bfg$ implies
\[
D_{N;\bfb,\bfg}(\omega) \le \sigma_N(\omega) \le D_{N;\bfb,\bfg}(\omega) + \epsilon/(4s).
\]
On the other hand,
Theorem \ref{theorem:discrepancy1} and Remark \ref{rem:bound1} yield
\[
D_{N;\bfb,\bfg}(\omega) \le D_N(\omega) \le D_{N;\bfb,\bfg}(\omega) + \epsilon/2.
\]
The result follows.
\end{proof}

\begin{corollary}
In the case of the star discrepancy,
put $\Lambda=\N^s$ and
let  $\myF$ be defined accordingly,
such that it contains all the functions $\myone_I - \lambda_s(I)$, $I\in \mathcal{J}_\bfb^*$.
For every $\epsilon>0$, there exists $\bfg\in\N^s$ such that
for any sequence $\omega$ in $[0,1)^s$,
the spectral test $\sigma_N(\omega)$ of the first $N$ elements of $\omega$,
with respect to $\myF$ and $\rho_\bfg$,
has the property
\[
|\sigma_N(\omega) - D_N^*(\omega)|< \epsilon.
\]
\end{corollary}

\bibliographystyle{degruyter-plain}


\def\cdprime{$''$} \def\cdprime{$''$} \def\cdprime{$''$}
\providecommand{\bysame}{\leavevmode\hbox to3em{\hrulefill}\thinspace}
\providecommand{\MR}{\relax\ifhmode\unskip\space\fi MR }
\providecommand{\MRhref}[2]{%
  \href{http://www.ams.org/mathscinet-getitem?mr=#1}{#2}
}
\providecommand{\href}[2]{#2}

\end{document}